\documentclass[10pt]{studiamnew}
\usepackage{graphicx}
\usepackage{amsmath}
\usepackage{amssymb}
\usepackage{url}
\usepackage{tree-dvips}
\usepackage{qtree}
\newtheorem{theorem}{Theorem}[section]
\newtheorem{lemma}[theorem]{Lemma}

\theoremstyle{definition}
\newtheorem{remark}[theorem]{Remark}

\newtheorem{problem}[theorem]{Problem}

\renewcommand{\theequation}{\thesection.\arabic{equation}}
\numberwithin{equation}{section}

\begin{document}
\setcounter{page}{1}
\setcounter{firstpage}{1}
\setcounter{lastpage}{4}
\renewcommand{\currentvolume}{??}
\renewcommand{\currentyear}{??}
\renewcommand{\currentissue}{??}
\title{A Subexponential Reduction from Product Partition to Subset Sum}
\author{Marius Costandin}
\email{costandinmarius@gmail.com}
%
%
%
\begin{abstract}
In this paper we study the Product Partition Problem (PPP), i.e. we are given a set of $n$ natural numbers represented on $m$ bits each and we are asked if a subset exists such that the product of the numbers in the subset equals the product of the numbers not in the subset.  Our approach is to obtain the integer factorization of each number. This is the subexponential step. We then form a matrix with the exponents of the primes and show that the PPP has a solution iff some Subset Sum Problems have a common solution. Finally, using the fact that the exponents are not large we combine all the Subset Sum Problems in a single Subset Sum Problem (SSP) and show that its size is polynomial in $m,n$. We show that the PPP has a solution iff the final SSP has one.
\end{abstract}
\maketitle

%
%
%
%
%
%
%
%
%

\section{Introduction}

In this paper the well known Product Partition Problem (PPP) is discussed and some results are obtained for it, namely a subexponential reduction scheme to Subset Sum Problem (SSP). The subexponential step (but not polynomial) is due to factoring of the natural numbers given in the PPP problem. Once the integer factorization is done, all the remaining steps are polynomial. There are several main ideas exploited throughout the paper:
\begin{enumerate}
\item For the product partition problem, one needs an instance of each prime in each set.
\item The power of each prime should be equal in both sets
\item The exponents are small, i.e. they require at most $\log(m)$ bits 
\end{enumerate}  

The cases 1,2 allow the formulation of a set of intermediary subset sum problems which have a common solution iff the PPP has a solution. The case 3 allows combining the intermediary subset sums in one subset sub problem (SSP) of polynomial size.

\section{Main results}

For $n \in \mathbb{N}$, let $S \in \mathbb{N}^{n}$. For $S = \begin{bmatrix} s_1, \hdots, s_n \end{bmatrix}^T$ we assume that $s_i$ is represented on at most $m\in \mathbb{N}$ bits, hence $s_i \leq 2^m$. The Product Partition Problem (PPP) asks:
\begin{problem} \label{P2.1}
 Exists $\mathcal{C} \subseteq \{1, \hdots,n\}$ such that 
\begin{align}\label{E2.1}
\prod_{i\in \mathcal{C}} s_i = \prod_{i \in \{1, \hdots,n\} \setminus \mathcal{C}} s_i \hspace{1cm}?
\end{align} 
\end{problem}

Since interger factoring is known to be subexponential \cite{intF}, the prime factors of the numbers in $S$ can be obtained in subexponential time. Let $\{p_1, \hdots, p_q\}$ be the  set of all the prime numbers involved, in ascending order. As such it is obtained
\begin{align}
s_i = \prod_{k=1}^q p_k^{\alpha_{ik}} \hspace{1cm} \forall i \in \{1, \hdots,n\}
\end{align} where $\alpha_{ik} \in \{0\} \cup \mathbb{N}$.

Construct the Table \ref{Ta1}. 

\begin{table}[h] 
\begin{tabular}{ |p{1cm}|p{1cm}|p{1cm}|p{1cm}|  }
 \hline
 \multicolumn{4}{|c|}{general representation} \\
 \hline
 $ $& $p_1$ & $\hdots$ &$p_q$\\
 \hline
$s_1$&   $\alpha_{11}$ & $\hdots$  &$\alpha_{1q}$\\
 $\vdots$   & $\vdots$    &$\ddots$ &$\vdots$ \\
 $s_n$ &$\alpha_{n1}$ &$\hdots$ &$\alpha_{nq}$\\
 \hline
\end{tabular}
\caption{Matrix representation of the prime powers involved}
\label{Ta1}
\end{table}

From Table \ref{Ta1} define the matrix
\begin{align}\label{E2.3}
S_M = \begin{bmatrix} \alpha_{11} &\hdots &\alpha_{1q}\\ \vdots &\ddots &\vdots\\ \alpha_{n1} &\hdots &\alpha_{nq}\end{bmatrix} \in \mathbb{N}^{n \times q}.
\end{align}

The following result can be stated:
\begin{lemma}\label{L2.1}
The PPP has a solution iff exists $x \in \{0,1\}^n$ such that 
\begin{align}\label{E2.4}
\left(x - \frac{1}{2}\cdot 1_{n\times 1}\right)^T\cdot S_M = 0_{1 \times q}
\end{align}
\end{lemma}
\begin{proof}
In order to have an equal product in (\ref{E2.1}), one needs every prime in each side to have equal power. 

Let $\mathcal{C} \subseteq \{1, \hdots, n\}$ be a solution to PPP, then define $x^T\cdot e_i = 1$ for all $i \in \mathcal{C}$ and zero otherwise. On the other hand, let $x$ be a solution to (\ref{E2.4}), then let $\mathcal{C} = \{ i\in \{1, \hdots, n\} | x^T \cdot e_i = 1\}$. 
\end{proof}

Let $\alpha_{\cdot,k} \in \mathbb{N}^{n}$ denote the column $k$ in $S_M$ for all $k \in \{1, \hdots, q\}$ and  $\alpha_{i,\cdot} \in \mathbb{N}^q$ denote the line $i$ in $S_M$ for all $i \in \{1, \hdots, n\}$. 
\begin{remark}\label{R2.2}
According to Lemma \ref{L2.1} solving the PPP is equivalent with finding $x \in \{0,1\}^n$ with 
\begin{align}
x^T\cdot \alpha_{\cdot,k} = \frac{1}{2}\cdot 1_{n \times 1}^T \cdot \alpha_{\cdot, k} \hspace{0.5cm} \forall k \in \{1, \hdots, q\}
\end{align} i.e. solving simultaneously $q$ Subset Sum Problems. 
\end{remark}

 However, solving simultaneously $q$ subset sum problems is known to be strongly NP-complete, hence it can not have, in general, a polynomial reduction to the Subset Sum Problem, which is known to be weakly NP-Complete, unless P = NP. 

One can take advantage here on the following observation:
\begin{remark} \label{R2.3}
The elements of $S_M$ are not large! Indeed, since they are actually the exponents of the primes forming each number, if each number is represented on $m$ bits, hence is smaller than $2^m$ and each prime is at least $2$, follows that in particular $\alpha_{ik} \leq m$ for all $i \in \{1, \hdots, n\}$ and $k \in \{1, \hdots, q\}$. 
\end{remark}
\subsection{Classical Analytical Analysis}
Next, consider the following result:
\begin{lemma}\label{L2.4}
Let $N \in \mathbb{N}$ large enough, $q \in \mathbb{N}$ fixed and $y_k \in \left[-\frac{N}{2},\frac{N}{2} \right] \cap \mathbb{Z}$ for all $k \in \{1, \hdots, q\}$. Then one has
\begin{align}
y_1 + y_2 \cdot 2 \cdot N + \hdots + y_q \cdot q \cdot N^{q-1} = 0 \iff y_k = 0 \hspace{0.5cm} \forall k 
\end{align}
\end{lemma}
\begin{proof}
We show that 
\begin{align}
\left| y_1 + y_2 \cdot 2 \cdot N + \hdots + y_{q-1} \cdot (q-1) \cdot N^{q-2} \right| < q \cdot N^{q-1}.
\end{align} Indeed
\begin{align}\label{E2.8}
&\left| y_1 + y_2 \cdot 2 \cdot N + \hdots + y_{q-1} \cdot (q-1) \cdot N^{q-2} \right| \leq \nonumber \\
& \frac{N}{2} \cdot \left(1 + 2\cdot N + \hdots + (q-1)\cdot N^{q-2} \right)
\end{align}

Next, for some $x \in \mathbb{R}$, consider the polynomial:
\begin{align}
P'(x) &= 1 + 2\cdot x + 3 \cdot x^2 + \hdots + (q-1) \cdot x^{q-2} \nonumber \\
&= \frac{d}{dx} P(x) 
\end{align} where 
\begin{align}
P(x) = x + x^2 + \hdots + x^{q-1} = \frac{x^q - 1}{x - 1} - 1
\end{align}  hence
\begin{align}\label{E2.11}
P'(x) &= \frac{d}{dx} \left(\frac{x^q - 1}{x - 1} - 1 \right) = \frac{q \cdot x^{q-1} (x-1) - \left(x^q - 1\right)}{(x-1)^2} \nonumber \\
& = \frac{(q-1)\cdot x^{q} - q\cdot x^{q-1} + 1}{(x-1)^2}
\end{align} From (\ref{E2.8}) and (\ref{E2.11}) one gets
\begin{align}\label{E2.12}
 N \cdot P'(N) &= \frac{N}{N-1} \cdot \left( (q-1) \cdot \frac{N^q}{N-1} - q \cdot \frac{N^{q-1}}{N-1} + \frac{1}{N-1} \right) \nonumber \\
& = \frac{N^2}{(N-1)^2}  \cdot (q-1)\cdot N^{q-1} - \frac{N^2}{(N-1)^2}\cdot  q \cdot N^{q-2} + \frac{N}{(N-1)^2}
\end{align} Since,

\begin{align}
\frac{N^2}{(N-1)^2} = \frac{N^2 - 2 \cdot N + 1 }{(N-1)^2} + \frac{2 \cdot N - 1}{(N-1)^2} = 1 +\frac{2 \cdot N - 1}{(N-1)^2} 
\end{align} the equation (\ref{E2.12}) becomes
\begin{align}
N \cdot P'(N) =& (q-1)\cdot N^{q-1} + \frac{N}{(N-1)^2}\nonumber \\
& + \left( \frac{2\cdot N^2 - N}{(N-1)^2} \cdot (q-1) - \frac{N^2}{(N-1)^2}\cdot q\right) \cdot N^{q-2} 
\end{align} which gives the coefficinet of $N^{q-2}$ as the middle paranthesis
\begin{align}
\frac{N}{N-1} \cdot q - \frac{2\cdot N^2 - N}{(N-1)^2} \to^{N\to \infty} q - 2 
\end{align} It follows that

\begin{align}
\frac{N \cdot P'(N)}{(q-1)\cdot N^{q-1} + (q-2)\cdot N^{q-2}} \to^{N\to\infty} 1
\end{align} Finally
\begin{align}
\frac{\frac{N}{2} \cdot P'(N)}{q \cdot N^{q-1}} &= \frac{(q-1)\cdot N^{q-1} + (q-2)\cdot N^{q-2}}{q \cdot N^{q-1}} \cdot  \frac{\frac{N}{2}\cdot P'(N)}{(q-1)\cdot N^{q-1} + (q-2)\cdot N^{q-2}} \nonumber \\
& \to^{N\to \infty} \frac{q-1}{q} \cdot \frac{1}{2} < 1
\end{align}
\end{proof}

Lemma \ref{L2.4} suggests a method of combining the $q$ Subset Sums from Remark \ref{R2.2} into one subset sum as follows. Let $N = n \cdot m$ and define the numbers:
\begin{align}
\hat{s}_i = \sum_{k=1}^q \alpha_{i,k}\cdot k \cdot (n\cdot m)^{k-1} \hspace{0.5cm} \forall i \in \{1, \hdots, n\}
\end{align} then the set 
\begin{align}
\hat{S} = \begin{bmatrix} \hat{s}_1 &\hdots &\hat{s}_n \end{bmatrix}^T \in \mathbb{N}^n
\end{align}

Define a target for the Subset Sum as follows:
\begin{align}\label{E2.20a}
\hat{T} =  \sum_{k=1}^q T_k \cdot k \cdot (n\cdot m)^{k-1} 
\end{align} where  $T_k = \sum_{i = 1}^n \frac{\alpha_{i,k}}{2} = \frac{1}{2} \cdot 1_{n\times1}^T \cdot \alpha_{\cdot,k}$. These are assumed natural numbers, otherwise the PPP is infeasible. 

We can take a moment to acknowledge that the memory requirements for storing $\hat{S}$ is indeed polynomial. Indeed
\begin{align}
\left| \hat{s}_i \right| \leq q \cdot |\alpha_{i,k}| \cdot q \cdot (2 \cdot n\cdot m)^q \leq q^2 \cdot m \cdot (m\cdot n)^q
\end{align} since $|\alpha_{i,k}| \leq m$ as showed in Remark \ref{R2.3}. Taking logarithm of $|\hat{s}_{i}|$ gives the required number of bits
\begin{align}
\log_{2}(|\hat{s}_{i}|) \leq 2\cdot \log_2(q) + \log_{2}(m) +  q \cdot \log_2( m \cdot n)
\end{align} We address now the quantity $q$, i.e. the number of prime numbers needed to write the original product $\prod_{i=1}^n s_i$. We show that this number is a polynomial in $n,m$. Indeed, since the original numbers $s_i$ were represented on $m$ bits, their maximum range is bounded above by $2^m$. By the use of Prime Number Theorem, there are on average, at most $\log(2^m) =  m \cdot \log(2)$ prime numbers smaller than $s_i$, therefore the maximum numbers of prime factors  $s_i$ can have is less that $m\cdot 2$ (on average). Therefore, $q$ is at most $2\cdot m \cdot n$.  In a similar way one can prove that $\hat{T}$ requires polynomial memory.

The obvious now Subset Sum Problem is:
\begin{problem} \label{P2.5}
Exits $x \in \{0,1\}^n$ such that 
\begin{align}
x^T \cdot \hat{S} = \hat{T} \hspace{1cm}?
\end{align} 
\end{problem} 

Next, we use Lemma \ref{L2.4} to show that Problem \ref{P2.5} has a solution iff   Problem \ref{P2.1} has one. The main result of this paper is:
\begin{theorem}\label{T2.7}
Problem \ref{P2.1} has a solution iff Problem \ref{P2.5} has one.
\end{theorem}
%

%
\begin{proof}
First, it is obvious that if the Product Partition Problem (PPP) has a solution then the SSP problem has a solution as well. For this use Lemma \ref{L2.1} and the Remark \ref{R2.2} to obtain $x \in \{0,1\}^n$ which solves simultaneously the $q$ subset sums given in Remark \ref{R2.2}.  Then

\begin{align}\label{E2.25}
x^T\cdot \hat{S} &= \sum_{i=1}^n x_i \cdot \sum_{k=1}^q \alpha_{i,k} \cdot k \cdot (m \cdot n)^{k-1} \nonumber \\
& = \sum_{k=1}^q k \cdot (m \cdot n)^{k-1} \cdot \sum_{i=1}^n x_i \cdot \alpha_{i,k} \nonumber \\
& =  \sum_{k=1}^q k \cdot (m \cdot n)^{k-1} \cdot x^T \cdot \alpha_{\cdot , k}
\end{align} but because $x$ simultaneously solves the SSPs in Remark \ref{R2.2} follows that $x^T \cdot \alpha_{\cdot , k} = \frac{1}{2}\cdot 1_{n \times 1}^T \cdot \alpha_{\cdot,k} $ hence from (\ref{E2.25}) follows that $x^T \cdot \hat{S} = \hat{T} $ i.e. solves the Problem \ref{P2.5}.

On the other hand, let $x$ be a solution to Problem \ref{P2.5} then $x^T\cdot \hat{S} = \hat{T}$, hence from (\ref{E2.25}) and the definition of $\hat{T}$ in (\ref{E2.20a}) one gets:
\begin{align}\label{E2.26}
0 = x^T\cdot \hat{S} - \hat{T} =  \sum_{k=1}^q k \cdot (m \cdot n)^{k-1} \cdot \left( x^T \cdot \alpha_{\cdot , k} - \frac{1}{2} \cdot 1_{n\times 1}^T \cdot \alpha_{\cdot, k}\right).
\end{align} Note that since $m \geq \alpha_{i, k} \geq 0$ one has
\begin{align}
 -\frac{1}{2} \cdot 1_{n\times 1}^T \cdot \alpha_{\cdot, k} & \leq x^T \cdot \alpha_{\cdot , k} - \frac{1}{2} \cdot 1_{n\times 1}^T \cdot \alpha_{\cdot, k} \leq \frac{1}{2} \cdot 1_{n\times 1}^T \cdot \alpha_{\cdot, k} \nonumber \\
 -\frac{m\cdot n}{2} &\leq x^T \cdot \alpha_{\cdot , k} - \frac{1}{2} \cdot 1_{n\times 1}^T \cdot \alpha_{\cdot, k} \leq \frac{m \cdot n}{2}.
\end{align}  Taking $N = n \cdot m$ in Lemma \ref{L2.4} follows that in (\ref{E2.26}) one has 
\begin{align}
x^T \cdot \alpha_{\cdot , k} - \frac{1}{2} \cdot 1_{n\times 1}^T \cdot \alpha_{\cdot, k} = 0
\end{align} for all $k \in \{1, \hdots, q\}$ hence $x$ simultaneously solves the SSPs in Remark \ref{R2.2} and hence solves Problem \ref{P2.1}.
\end{proof} 

\subsection{Numerical Analysis Improvements}
Note that the Problem \ref{P2.5} can be solved exactly using dynamical programming in  time and memory exponential in $q$ i.e. the number of distinct primes in the product $\prod_{i=1}^n s_i$. As a matter of fact, $q$ is actually the Prime omega function 
\begin{align}
q  := \omega\left(\prod_{i=1}^n s_i \right).
\end{align}

In this subsection, we want to improve the results from the previous subsection, namely the obtained Subset Sum Problem. For this, we give a stronger version of Lemma \ref{L2.4}. It's validity however, shall only be analyzed based on some numerical simulations.  We first note that a similar lemma to Lemma \ref{L2.4} is true

\begin{lemma}\label{L2.8}
For $q \in \mathbb{N}$ fixed, exists $N \in [1, 2]$ such that for all $y_k \in \left[-\frac{N}{2},\frac{N}{2} \right] \cap \mathbb{Z}$ 
one has
\begin{align}
y_1 + y_2 \cdot 2 \cdot N + \hdots + y_q \cdot q \cdot N^{q-1} = 0 \iff y_k = 0 \hspace{0.5cm} \forall k 
\end{align}
\end{lemma}

\begin{proof} 
We give here a proof based on a plot of  of numerical graphic of a certain polynomial quotient. As for Lemma \ref{L2.4}, we want to show that: 
\begin{align}
\left| y_1 + y_2 \cdot 2 \cdot N^L + \hdots + y_{q-1} \cdot (q-1) \cdot \left(N^L\right)^{q-2} \right| < \left| y_q\cdot q \cdot \left(N^L\right)^{q-1}\right| .
\end{align} for every $q$. 

Indeed, 
\begin{align}
&\left| y_1 + y_2 \cdot 2 \cdot N + \hdots + y_{q-1} \cdot (q-1) \cdot N^{q-2} \right| \leq \nonumber \\
& \frac{N}{2} \cdot \left(1 + 2\cdot N + \hdots + (q-1)\cdot N^{q-2} \right) = \frac{N}{2}\cdot P'\left(N\right)
\end{align}

Unfortunately, we only do this by analysing how the plot of $\frac{\frac{N}{2}\cdot P'(N)}{ q \cdot N^{L \cdot (q-1)}} $ looks for different values of $q$. We give the following representative figure
\begin{figure}[h]
\centering
\includegraphics[scale=0.55]{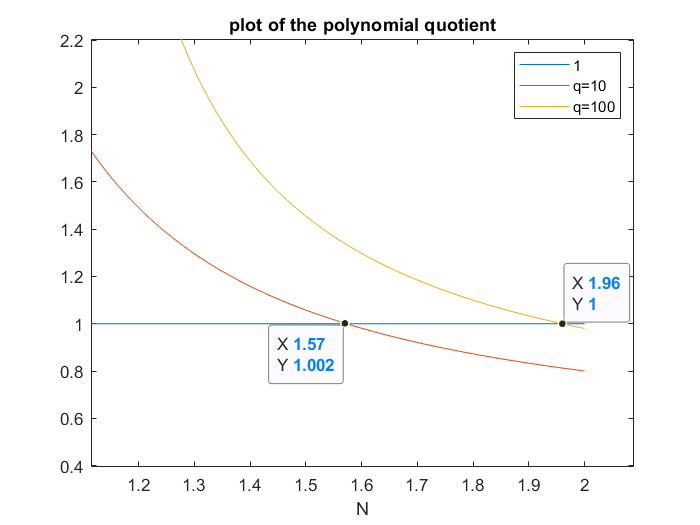}
\caption{In this figure one can see $\frac{\frac{N}{2}\cdot P'(N)}{ q \cdot N^{L \cdot (q-1)}} $ for different values of $q$.}
\label{Fig1}
\end{figure}

Different values of $q$ will generate different plots. The higher the $q$ is the closer the intersection of the quotient to the $1$ line is. However, it seems that the intersection always happens before the abscisa equals $2$. 
\end{proof}

Lemma \ref{L2.8} suggests a different method of combining the $q$ Subset Sums from Remark \ref{R2.2} into one subset sum as follows.  Define the numbers:
\begin{align}
\hat{s}_i = \sum_{k=1}^{q} \alpha_{i,k}\cdot k \cdot N(q)^{k-1} 
\end{align} then the set 
\begin{align}
\hat{S} = \begin{bmatrix} \hat{s}_1 &\hdots &\hat{s}_n \end{bmatrix}^T \in \mathbb{N}^n
\end{align} where $1 \leq N(q) \leq 2$ is the intersection of the quotient to the line $1$ in Figure \ref{Fig1}.

Define a target for the Subset Sum as follows:
\begin{align}\label{E2.20a}
\hat{T} =  \sum_{k=1}^{q} T_k \cdot k \cdot N(q)^{k-1} 
\end{align} where  $T_k = \sum_{i = 1}^n \frac{\alpha_{i,k}}{2} = \frac{1}{2} \cdot 1_{n\times1}^T \cdot \alpha_{\cdot,k} \in \mathbb{N}$ are assumed natural, otherwise the PPP is infeasible.  
%
%
%

Some final remarks are presented below:

\begin{remark} The size of the numbers involved, $\hat{s}_i$, is still exponential in $q$, but the base of the exponent is a rational number smaller than $2$. However, if the numbers in the final SSP, we proved that $N(q) = 2$ is the smallest base for the exponent we can take in general.  
\end{remark}

Finally we relax the PPP in the following way: 

\begin{problem}
Given $S \in \mathbb{N}^n$, exists $K \in \mathbb{N}^n $ and $\mathcal{C} \subseteq \{1, \hdots, n\}$such that
\begin{align}
\prod_{i \in \mathcal{C}} s_i^{k_i} = \prod_{i\in \{1, \hdots, n\}\setminus \mathcal{C}}s_i^{k_i} \hspace{0.5cm}?
\end{align}
\end{problem}

It can be shown that the above relaxation has a solution if exists $x \in \mathbb{Q}^n$ such that 
\begin{align}
x^T \cdot S_M = 0_{1\times q}
\end{align} with $x^T\cdot e_i \neq 0$ for all $i \in \{1, \hdots, n\}$ where $S_M$ is given by (\ref{E2.3}).

\section{Conclusion and future work}
 
Given a Product Partition Problem we showed that it can be reduced in sub-exponential time to an instance of Subset Sum Problem. The size of the obtained SSP problem is showed to be polynomial in the size of the initial PPP problem. The subexponential step is due to the factorization of the involved natural numbers in the original PPP problem. This shows that if integer factorization would pe polynomial, then the PPP could be solved by solving the SSP.

\end{document}